\documentclass[12pt]{amsart}

\usepackage[english]{babel}
\usepackage[utf8x]{inputenc}
\usepackage[T1]{fontenc}
\usepackage{todonotes}
\usepackage{tikz}
\usepackage{mathtools}

\usepackage[a4paper,top=3cm,bottom=3cm,left=3cm,right=3cm,marginparwidth=1.75cm]{geometry}

\usepackage{varioref}
\usepackage{hyperref}
\usepackage[nameinlink, capitalize, noabbrev]{cleveref}
\usepackage{todonotes}

\usepackage{graphicx}
\usepackage{amsmath}
\usepackage{xypic}
\usepackage{amssymb}
\theoremstyle{plain} 

\newtheorem{thm}{Theorem}[section]
\newtheorem{lemma}[thm]{Lemma}
\newtheorem{prop}[thm]{Proposition}
\newtheorem{corl}[thm]{Corollary}

\newtheorem{exmp}[thm]{Example}

\theoremstyle{definition} 
\newtheorem{defn}[thm]{Definition}

\theoremstyle{remark} 
\newtheorem{rmk}[thm]{Remark}
\numberwithin{equation}{section}

\DeclareMathOperator{\Iso}{Iso} 

\newcommand{\A}{\mathcal{A}}		
\newcommand{\B}{\mathcal{B}}		
\newcommand{\Cb}{\mathcal{C}}		
\newcommand{\R}{\mathbb{R}}			
\newcommand{\C}{\mathbb{C}}			
\newcommand{\Z}{\mathbb{Z}}	
\newcommand{\N}{\mathbb{N}}			
\newcommand{\T}{\mathbb{T}}			

\newcommand{\G}{\mathcal{G}}		    
\newcommand{\Go}{\mathcal{G}^{(0)}}		

\newcommand{\Spec}{\mathrm{Sp}}	

\title[Groupoids and Hermitian $*$-Banach algebras]{Groupoids and Hermitian Banach $*$-algebras}
\author{Are Austad and Eduard Ortega}

\begin{document}
	\maketitle
	
	\begin{abstract}
		We study when the twisted groupoid Banach $*$-algebra $L^1(\G,\sigma)$ is Hermitian. 
		In particular, we prove that Hermitian groupoids satisfy the weak containment property. Furthermore, we find that for $L^1(\G,\sigma)$ to be Hermitian it is sufficient that $L^1 (\G_\sigma)$ is Hermitian. Moreover, if $\G$ is ample, we find necessary conditions for $L^1(\G,\sigma)$ to be Hermitian in terms of the fibers $\G^x_x$. 
	\end{abstract}

	\section{Introduction}

  Due to the seminal paper \cite{Gel}, Hermitian Banach $*$-subalgebras of $C^*$-algebras have been intimately linked to Wiener's lemma and the notion of spectral invariance. As such, they appear in a variety of areas of mathematics, such as approximation theory, time-frequency analysis and signal processing and noncommutative geometry, see e.g.\ \cite{Bill, Gro,GoKl,GoKl2,GoLe}. 


We say that a locally compact group $G$ is Hermitian if the convolution algebra $L^1 (G)$ is a Hermitian Banach $*$-algebra. In the realm of Hermitian locally compact groups the literature is very extensive (see for example \cite{Palma,Palmer}). The class of Hermitian locally compact groups includes compact extensions of nilpotent groups \cite{Ludwig79} and famously also all compactly generated groups of polynomial growth \cite{Losert}. In \cite{SaWi} it was proved that a Hermitian locally compact group must be amenable.  
In another, but related context, it was proved in \cite{LepPo} that if $\A$ is a Hermitian Banach $*$-algebra and $\Gamma$ is a compact group acting on $\A$ by $*$-automorphisms, the generalized $L^1$-algebra $L^1(\Gamma,\A)$ is Hermitian.

One of the main motivations for this article is the use of spectral invariance of  noncommutative tori in time-frequency analysis as in e.g.\ \cite{GoLe}. Noncommutative tori can be viewed as the twisted group $C^*$-algebra $C^*(\Z^{2d},\sigma)$, where $\sigma$ is the Heisenberg $2$-cocycle. However, this can generalized to the group $C^*$-algebra  $C^*(\Delta,\sigma)$ where $\Delta$ is a closed subgroup of the time-frequency plane $G\times\widehat{G}$, where $G$ is a locally compact abelian group. This extension has been used to construct finitely generated modules over noncommutative solenoids  \cite{EJL,LaPa}. Spectral invariance of $L^1(\Delta,\sigma)$  was studied in \cite{Austad20}. Pushing forward this generalization, given a quasicrystal $\Lambda\subseteq \R^{2d}$ one can  construct a twisted groupoid $C^*$-algebra $C^*(\G_{\Lambda},\sigma)$ where $\G_\Lambda$ is an étale groupoid. Finitely generated projective modules over $C^*(\G_\Lambda, \sigma)$ were constructed in \cite{Kre}.


This paper together with \cite{AusOrt} is a starting point for a project  that aims to extend the tools of operator algebras applied to Gabor analysis for lattices in the time-frequency plane (see for example in \cite{AuEn20, EJL, Luef,Rieff}) to the more general context of quasicrystals \cite{BrMe,Kre}. In particular, in this paper  we focus our attention on 
when the Banach $*$-algebra $L^1(\G,\sigma)$ associated to a locally compact groupoid $\G$ with a twist $\sigma$ is Hermitian. 

The paper is structured as follows.
First, in Section \ref{sec:prelims} we give the necessary background and results about groupoids and Banach $*$-algebras.


In Section \ref{quasi-her-sec} we prove the main result of the paper, namely,  if $L^1(\G,\sigma)$  is Hermitian, then the full and the reduced twisted groupoid $C^*$-algebras coincide  (also known as the weak containment property). When $\G$ is a group, this is equivalent to amenability. However, in the case of locally compact groupoids  the situation is more subtle. 


Then, in Section \ref{sec_non_herm} we give prove that for ample groupoids a necessary condition for  $L^1(\G,\sigma)$ to be Hermitian is that the "fibers'' $L^1(\G_x^x,\sigma_x)$ are Hermitian for every $x\in \Go$. Using this we give a simple 
example of an amenable étale groupoid such that $L^1(\G,\sigma)$ is not Hermitian.  

Finally, in Section \ref{sec_exa_herm_twist} we give sufficient conditions for 	$L^1(\G,\sigma)$ to be Hermitian. We construct the twisted groupoid $\G_\sigma$, and show that $L^1(\G,\sigma)$ is Hermitian if the "untwisted'' groupoid Banach $*$-algebra $L^1(\G_\sigma)$ is Hermitian. In particular, we prove that if $\Gamma$ is a compact group or a locally compact abelian group acting by homeomorphisms on a locally compact Hausdorff space $X$, then $L^1(X\rtimes \Gamma,\sigma)$ is Hermitian for every group $2$-cocycle $\sigma$ of $\Gamma$, where $X\rtimes \Gamma$ is the transformation groupoid.

	\section{Preliminaries}\label{sec:prelims}
	
	\subsection{Hermitian Banach $*$-algebras}\label{sec:prelims-C*-uniqueness}

	If $\A$ is a unital Banach algebra, we denote by $\A^{-1}$ the set of invertible elements of $\A$. Given a unital Banach algebra $\A$   we denote by 
	$$\Spec_\A(a)=\{\lambda\in \C:  \lambda 1_\A - a \notin \A^{-1}\}\,,$$ 
	the \emph{spectrum of $a$} in $\A$, and by 
	$$r_\A(a):=\sup \{|\lambda|: \lambda \in \Spec_\A(a) \}=\lim_{n\to \infty}\|f^n\|^{1/n}$$
	 the \emph{spectral radius} of $a$. If $\A$ is not unital, given $a\in \A$ we define  $\Spec_\A(a):=\Spec_{\A^+}(a)$ and $r_\A(a)=r_{\A^+}(a)$, where $\A^+$ is the minimal unitization of $\A$.

	Now let $\A$ be a Banach $*$-algebra. Then $a\in \A$ is  \emph{Hermitian} if $a^*=a$. If $\mathcal{S}$ is a subset of $\A$, we denote by $\mathcal{S}_h=\{a\in \mathcal{S}: a^*=a\}$ the set of Hermitian elements in $\mathcal{S}$. 
	
	\begin{defn}
		A Banach $*$-algebra  $\A$  is \emph{Hermitian} if 
		$$\Spec_\A(a)\subseteq \R$$
		for every $a\in \A_h$. $\A$ is called \emph{symmetric} if
		$$\Spec_\A(aa^*)\subseteq [0,\infty)$$
				for every $a\in \A$.
	\end{defn}
	\begin{rmk}
By the celebrated Shirali-Ford theorem, a Banach $*$-algebra is Hermitian if and only if it is symmetric.
	\end{rmk}
	A \emph{$*$-representation} of a Banach $*$-algebra $\A$ is a $*$-homomorphism $\pi \colon \A \to B (\mathcal{H})$, where $B(\mathcal{H})$ denotes the bounded linear operators on a Hilbert space $\mathcal{H}$. We say $\A$ is \emph{reduced} 
	if $\A_\mathcal{R}=\{a\in\A: \pi(a)=0\text{ for every }*-\text{representation } \pi \text{ of }\A\}=\{0\}$. All Banach $*$-algebras we consider in the sequel will be \emph{reduced}. 
	The \emph{enveloping $C^*$-algebra} of a reduced 
	Banach $*$-algebra $\A$ is the completion of $\A$ with respect to the norm 
	$$\|a\|:=\sup\{\|\pi(a)\|: \pi:\A\to B(\mathcal{H}) \text{ is a $*$-representation} \}$$ 
	for every $a\in \A$, and it is denoted by $C^*(\A)$. The enveloping $C^*$-algebra of a Banach $*$-algebra always exists \cite[Section 10.1]{Palmer}.

	\subsection{Invariant spectral radius}
	
	\begin{defn}
		We say that $\A\subseteq \B$ is a \emph{nested pair} of reduced 
		Banach $*$-algebras if $\A$ and $\B$ are reduced Banach $*$-algebras and $\A$ embeds continuously into $\B$ as a dense $*$-subalgebra. A \emph{nested triple} of reduced 
		Banach $*$-algebras is defined similarly.
	\end{defn}
	
	\begin{defn}
		Let $\A\subseteq \B$ be a nested pair of reduced Banach $*$-algebras and $\mathcal{S}$ a (not necessarily closed) $*$-subalgebra of $\A$. We say $\mathcal{S}_h$ has \emph{invariant spectral radius} in $(\A,\B)$  if 
		$$r_\A(a)=r_\B(a)\,,$$
		for every $a\in \mathcal{S}_h$. If $\A_h$ has invariant spectral radius in $(\A,\B)$, we say that $\A_h$ has \emph{invariant spectral radius} in $\B$. 
		Moreover, we say $\mathcal{S}$ is a spectrally invariant subalgebra of $(\A, \B)$ if
		$$\Spec_\A(a)=\Spec_\B(a)\,,$$
		for every $a \in \mathcal{S}$. If $\A$ is a spectrally invariant subalgebra of $(\A, \B)$, we say $\A$ is spectrally invariant in $\B$. 
	\end{defn}

	Clearly, if $\mathcal{S}$ is a spectrally invariant subalgebra of 
	$(\A,\B)$, then $\mathcal{S}_h$ has invariant spectral radius in $(\A,\B)$. The Barnes-Hulanicki Theorem \cite{Ba90} provides a partial converse when $\B$ is a $C^*$-algebra.

\begin{thm}\label{Hulanicki}
Let $\A$ be a Banach $*$-algebra, $\mathcal{S}$
 a $*$-subalgebra of $\A$, and $\pi:\A\to B(\mathcal{H})$ a faithful $*$-representation. If $\A$ is unital, we assume that $\pi(1_\A)=1_{B(\mathcal{H})}$. If 
$$r_\A(a)=\|\pi(a)\|\,,$$ for all $a\in \mathcal{S}_h$, then 
\begin{equation*}
	\Spec_\A(a)=\Spec_{\B(\mathcal{H})}(\pi(a))
\end{equation*}
for every $a\in \mathcal{S}$.	
\end{thm}

In \cite[Proposition 2.7]{SaWi} they prove a very useful property 
related to invariant spectral radius.
\begin{prop}\label{same_C_env}
	Let $\A\subseteq \B$ be a nested pair of reduced 
	Banach $*$-algebras, and let $\mathcal{S}$ be a dense $*$-subalgebra of $\A$. Suppose that $\mathcal{S}_h$ has invariant spectral radius in $(\A,\B)$. Then $\A$ and $\B$ have the same $C^*$-envelope. 
\end{prop}

\subsection{Hermitian and quasi-hermitian Banach $*$-algebras}

Let $\A$ be a reduced 
Banach $*$-algebra. Then the following statements are equivalent:
\begin{enumerate}
	\item $\A$ is Hermitian,
	\item $\A$ is symmetric, 
	\item $\A$ is spectrally invariant in $C^* (\A)$,
	\item $r_\A(a)=r_{C^*(\A)}(a)$ for every $a\in \A$.
\end{enumerate}
(see \cite[Lemma 2.8]{SaWi} and \cite[p.\ 340]{Li12}).

\begin{defn} A dense $*$-subalgebra $\mathcal{S}$ of a Banach $*$-algebra $\A$ is called \emph{quasi-Hermitian} in $\A$ if $\Spec_\A(a)\subseteq \R$ for every $a\in \mathcal{S}_h$. We say $\mathcal{S}$ is \emph{quasi-symmetric} in $\A$ is $\Spec_\A(a^*a)\subseteq [0,\infty)$ for every $a \in \mathcal{S}$.
\end{defn}

Let $\mathcal{S}$ be a  dense $*$-subalgebra of a Banach $*$-algebra $\A$. If $\A$ is Hermitian, then $\mathcal{S}$ is automatically quasi-Hermitian in $\A$. 
The converse is not true in general, but holds whenever $\A$ is commutative \cite[Proposition 2.10]{SaWi}. 

\subsection{Spectral interpolation}

\begin{defn}
	Suppose $\A\subseteq \B\subseteq \Cb$ is a nested triple of reduced Banach $*$-algebras and $\mathcal{S}$ is a dense $*$-subalgebra of $\A$. We say that $(\A,\B,\Cb)$ is a \emph{spectral interpolation triple} relative to $\mathcal{S}$ if there exists $\theta\in (0,1)$ such that 
	$$r_\B(a)\leq r_\A(a)^{1-\theta} r_\Cb(a)^\theta$$
	for every $a\in \mathcal{S}_h$.
\end{defn}

A nice property 
of spectral interpolation triples of reduced 
Banach $*$-algebras was proved in \cite[Propostion 3.4]{SaWi}.

\begin{prop}\label{sp_int_triple}
If $(\A,\B,\Cb)$ is a spectral interpolation triple of reduced 
Banach $*$-algebras relative to a quasi-Hermitian dense $*$-subalgebra $\mathcal{S}$ of $\A$, then $\mathcal{S}_h$ has invariant spectral radius in $(\B,\Cb)$. In particular, $C^*(\B)=C^*(\Cb)$ and 
$$r_{\B}(a)=r_\Cb(a)\,,$$
for every $a\in \mathcal{S}_h$. 
\end{prop}

	\subsection{Groupoids, twists and associated algebras}\label{sec:prelim-groupoids}
	We now introduce the Banach $*$-algebras and $C^*$-algebras associated to a groupoid as described in \cite{Ren}.  
	Given a topological groupoid $\G$ we will denote by $\G^{(0)}$ its unit space and write $r,s:\G\to \G^{(0)}$ for the continuous range and source maps, respectively. We will also denote by $\G^{(2)}=\{(\alpha,\beta)\in \G\times \G: s(\alpha)=r(\beta) \}$ the set of \emph{composable elements}. $\G^{(2)}$ inherits the subspace topology from $\G\times \G$. Given $x\in \G^{(0)}$ we define $\G_x:=\{\gamma\in \G: s(\gamma)=x\}$ and $\G^x:=\{\gamma\in \G: r(\gamma)=x\}$.
	
	Let $\lambda=\{\lambda_x\}_{x\in \Go}$ be a \emph{Haar system for $\G$},  where  $\lambda_x$ are measures with support $\G_x$ such that 
	\begin{enumerate}
	\item for every $f\in C_c(\G)$, the function $x \mapsto \int_{\G_x} f(\gamma)d\lambda_x(\gamma)$ is continuous,
	\item for every $\eta\in \G$ and $f\in C_c(\G)$ we have that 
	$$\int_{\G_{r(\eta)}} f(\gamma\eta) d\lambda_{r(\eta)}(\gamma)= \int_{\G_{s(\eta)}} f(\gamma) d\lambda_{s(\eta)}(\gamma)\,.$$
	\end{enumerate}
	If $\G$ is a locally compact groupoid, then there always exist Haar systems for $\G$.
	
	A groupoid $\G$ is called \emph{étale} if the range map, and hence also the source map, is a local homeomorphism. In this case the sets $\G^x$ and $\G_x$ are discrete sets for every $x\in \G^{(0)}$, and the Haar system consists of counting measures. A subset $B$ of an \'etale groupoid $\G$ is called a \emph{bisection} if there is an open set $U\subseteq \G$ containing $B$ such that $r\colon U\to r(U)$ and $s\colon U\to s(U)$ are homeomorphisms onto open subsets of $\Go$. Second-countable étale groupoids have countable bases consisting of open bisections.

 The \emph{orbit of $x\in \Go$} is defined to be $\text{Orb}_\G(x)=\{r(\gamma):\gamma\in \G_x \}$.	The \emph{isotropy group of $x$} is given by $\G_x^x:=\G^x\cap \G_x=\{\gamma\in \G: s(\gamma)=r(\gamma)=x\}$, and the \emph{isotropy subgroupoid of $\G$} is the subgroupoid $\Iso(\G):=\bigcup_{x\in \G^{(0)}}\G_x^x$ with the relative topology from $\G$. 
	
	We will consider groupoid twists where the twist is implemented by a normalized continuous $2$-cocycle. To be more precise, let $\G$ be a  locally compact   groupoid. A \emph{normalized continuous $2$-cocycle} is then a continuous map $\sigma \colon \G^{(2)} \to \T$ satisfying
	\begin{equation} \label{eq:2-cocycle-unit-condition}
	\sigma (r(\gamma),\gamma) = 1 = \sigma (\gamma, s(\gamma))
	\end{equation}
	for all $\gamma \in \G$, and
	\begin{equation} \label{eq:2-cocycle-associativity-condition}
	\sigma (\alpha,\beta) \sigma (\alpha\beta,\gamma) = \sigma(\beta,\gamma) \sigma(\alpha,\beta\gamma)
	\end{equation}
	whenever $(\alpha,\beta),(\beta,\gamma)\in \G^{(2)}$. The set of normalized continuous $2$-cocycles on $\G$ will be denoted $Z^2 (\G,\T)$. Note that this is not the most general notion of a twist of a groupoid (see \cite[Chapter 5]{Sims}). 
	
	Let $\G$ be a locally compact groupoid with Haar system $\lambda$. We will define the $\sigma$-twisted convolution algebra $C_c (\G,\sigma)$ as follows: As a set it is just 
	$$C_c (\G,\sigma)=\{f:\G\to \C: f \text{ is continuous with compact support}\},$$ 
	but equipped with $\sigma$-twisted convolution product
	\begin{equation}\label{eq:twisted-conv}
	(f \star_\sigma g) (\gamma) = 
	\int_{\G_{s(\gamma)}} f(\gamma\mu^{-1}) g(\mu) \sigma (\gamma\mu^{-1},\mu)\, d\lambda_{s(\gamma)}(\mu), \quad \text{$f,g\in C_c (\G,\sigma)$, $\gamma \in \G$,}
	\end{equation}
	and $\sigma$-twisted involution
	\begin{equation}\label{eq:twisted-inv}
	f^{*_\sigma}(\gamma) = \overline{\sigma(\gamma^{-1},\gamma)} \overline{f(\gamma^{-1})}, \quad \text{$f \in C_c (\G,\sigma)$, $\gamma \in \G$.}
	\end{equation}
	We complete $C_c (\G,\sigma)$ in the "fiberwise $1$-norm'', also known as the $I$-norm, given by
	\begin{align}\label{eq:I-norm}
	\Vert f \Vert_I  & = \sup_{x \in \Go}\left\lbrace \max \left\lbrace  \int_{\G_x} \vert f(\gamma)\vert \,d\lambda_x (\gamma)  , \int_{ \G_{x}} \vert f(\gamma^{-1})\vert\,d\lambda_x (\gamma)  \right\rbrace \right\rbrace  \\
	 & = \sup_{x\in \Go}\left\lbrace \max \left\lbrace  \int_{\G_x} \vert f(\gamma)\vert \,d\lambda_x (\gamma) , \int_{ \G_{x}} \vert f^{*_\sigma}(\gamma)\vert\,d\lambda_x (\gamma)   \right\rbrace \right\rbrace
	\end{align}
	for $f \in C_c (\G,\sigma)$. Denote by $L^1 (\G,\sigma,\lambda)$ (or  $L^1(\G,\sigma)$ when there is no ambiguity on the Haar system) the completion of $C_c (\G,\sigma)$ with respect to the $I$-norm.  If $\G$ is an étale groupoid, then we denote it by $\ell^1 (\G,\sigma)$. 

	The \emph{(full) twisted groupoid $C^*$-algebra} $C^*(\G,\sigma,\lambda)$ (or  $C^*(\G,\sigma)$ when there is no ambiguity on the Haar system) is the completion of $C_c (\G,\sigma)$ in the norm
	\begin{equation*}
	\Vert f \Vert := \sup \{ \Vert \pi (f) \Vert : \text{$\pi$ is an $I$-norm bounded $*$-representation of $C_c (\G,\sigma)$}\},
	\end{equation*}
	for $f \in C_c (\G,\sigma)$.
	It was observed in \cite[Lemma 3.3.19]{Armstrong} that if $\G$ is a locally compact Hausdorff étale groupoid, then every $*$-representation of $C_c(\G,\sigma)$ is bounded by the $I$-norm. In addition, if $\G$ is a transformation groupoid (see Example \ref{example_trans}), every $*$-representation is bounded by the $I$-norm. In these cases, since we are completing with respect to a supremum over $*$-representations, $C^*(\G,\sigma)$ is just the $C^*$-envelope of $L^1 (\G,\sigma)$. 
	
	Now we will construct a faithful $*$-representation of $L^1(\G,\sigma)$  called the \emph{$\sigma$-twisted left regular representation}. In particular, we have that $L^1(\G,\sigma)$ is reduced. 
	The completion of the image of $L^1(\G,\sigma)$ under the $\sigma$-twisted left regular representation is called the \emph{$\sigma$-twisted reduced groupoid $C^*$-algebra of $\G$} and will be denoted $C^*_r (\G,\sigma,\lambda)$ (or  $C^*_r(\G,\sigma)$ when there is no ambiguity on the Haar system). Let $x \in \Go$. Then there is a representation $L^{\sigma,2}_x \colon C_c (\G,\sigma) \to B(L^2 (\G_x))$ (here $L^2 (\G_x)=L^2 (\G_x,\lambda_x)$) which is given by
	\begin{equation}\label{eq:2-twisted-left-reg-rep}
	\left( L^{\sigma,2}_x (f) \xi\right) (\gamma) = \int_{\G_x} \sigma (\gamma\mu^{-1},\mu)f(\gamma\mu^{-1}) \xi(\mu) \,d\lambda_{s(\gamma)}(\mu)
	\end{equation}
	for $f \in C_c (\G,\sigma)$, $\xi\in L^2(\G_x)$ and $\gamma \in \G_x$. 
		
	We then obtain a faithful $I$-norm bounded $*$-representation of $C_c (\G,\sigma)$ given by
	\begin{equation*}
	\bigoplus_{x\in \Go} L^{\sigma,2}_x \colon C_c(\G,\sigma) \to \bigoplus_{x\in \Go} B(L^2 (\G_x)) \subset B(\bigoplus_{x\in \Go} L^2 (\G_x)).
	\end{equation*}
	$C^*_r (\G,\sigma)$ is then the completion of $C_c(\G,\sigma)$ with respect of the image of  $C_c (\G,\sigma)$ under the $\sigma$-twisted left regular representation, so given $f\in C_c (\G,\sigma)$
	$$\|f\|_{r,2}:=\sup_{x\in \Go}\{\|L^{\sigma,2}_x(f) \|_{B(L^2 (\G_x))} \}\,.$$
	As the representation is $I$-norm bounded, it extends to a $*$-representation of $L^1(\G,\sigma)$, and $C^*_r(\G,\sigma)$ is also the $C^*$-completion of $L^1 (\G,\sigma)$ in the extended $*$-representation. 
	Moreover, 
	since $C^*(\G,\sigma)$ is the  completion of $C_c(\G,\sigma)$ with respect to the supremum of the $I$-bounded norms, the identity map on $C_c (\G,\sigma)$ extends to a (surjective) $*$-homomorphism $\pi:C^*(\G,\sigma)\to C^*_r(\G,\sigma)$. 
	

	\begin{defn}
		Let $\G$ be a  locally compact  groupoid  with Haar system $\lambda$ and let $\sigma \in Z^2(G,\T)$. We say that $\G$ twisted by $\sigma$ has the \emph{weak containment property} when the natural map $\pi \colon C^* (\G,\sigma)\to C^*_r (\G,\sigma)$ is an isomorphism.
	\end{defn}
	
	If $\G$ is an amenable groupoid  with Haar measure $\lambda$ \cite{AnDRen}, we have that  $C^*_r (\G,\sigma,\lambda) = C^* (\G,\sigma,\lambda)$ for every $\sigma \in Z^2 (\G, \T)$ \cite[Proposition 6.1.8]{AnDRen}, and hence $\G$ twisted by $\sigma$ has the  weak containment property for every $\sigma \in Z^2 (\G, \T)$. In \cite{Will} it was proved that amenability is not equivalent to having the weak containment property.  
	
	 \begin{rmk}
	 It was shown in \cite{Ren} that the full and reduced $C^*$-algebras don't depend, up to Morita equivalence, on the Haar system. Suppose $\lambda,\lambda'$ are two Haar systems for a locally compact groupoid $\G$, let $\sigma \in Z^2(G,\T)$, and suppose $C^*_r(\G,\sigma,\lambda)=C^*(\G,\sigma,\lambda)$. Then $C^*_r(\G,\sigma,\lambda)$ and $C^*_r(\G,\sigma,\lambda')$ are Morita equivalent, as are $C^*(\G,\sigma,\lambda)$ and $C^*(\G,\sigma,\lambda')$. However, it is not known to the authors if this also implies $C^*_r(\G,\sigma,\lambda')=C^*(\G,\sigma,\lambda')$, that is, if weak containment is independent of the Haar system.

 	 \end{rmk}

	\begin{exmp}\label{example_trans}
	Let $\Gamma$ be a locally compact  group with left Haar measure $\lambda$ and with unit $e$, and let us consider an action of $\Gamma$ by homeomorphisms on a locally compact Hausdorff space $X$. Then we define the \emph{transformation groupoid} $X\rtimes \Gamma$ as the set $X\times \Gamma$ with the product topology, such that 
		$$s(x,\gamma)=(x,e)\,,\qquad r(x,\gamma)=(\gamma\cdot x,e)\qquad\text{and}\qquad (\gamma_1\cdot x,\gamma_2)(x,\gamma_1)=(x,\gamma_2\gamma_1)\,.$$
		Then $X\rtimes \Gamma$ is a locally compact groupoid. Moreover, if $X$ and $\Gamma$ are both second-countable, then so is $X\rtimes \Gamma$.
	One defines the Haar system $\{\delta_x\times \lambda\}_{x\in \Go}$ where $\delta_x$ is the Dirac measure, and in this case, given $f\in C_c(X\rtimes \Gamma)$ we have that 
		\begin{align*} \|f\|_I & =\sup\left\lbrace  \max \left\lbrace \int_{\Gamma} |f(x,\gamma)|\,d\lambda(\gamma)\,,\int_{\Gamma} |f(\gamma\cdot x,\gamma^{-1})|\,d\lambda(\gamma)\right\rbrace : x\in \Go  \right\rbrace  \\
		& =\sup\left\lbrace    \int_{\Gamma} |f(x,\gamma)|\,d\lambda(\gamma)  : x\in \Go  \right\rbrace 
		 \,.\end{align*}
	
	Now let $\sigma\in Z^2 (\Gamma,\T)$. Then we can extend $\sigma$ to a $2$-cocycle of $X\rtimes \Gamma$, also denoted by $\sigma$, by defining
		$$\sigma((x_1,\gamma_1),(x_2,\gamma_2)):=\sigma(\gamma_1,\gamma_2)\,,$$
		for all $x_1,x_2\in X$ and $\gamma_1,\gamma_2\in \Gamma$ for which $((x_1,\gamma_1), (x_2, \gamma_2)) \in (X\rtimes \Gamma)^{(2)}$. 
	Then $$C^*(L^1(X\rtimes \Gamma,\sigma))\cong C^*(X\rtimes \Gamma,\sigma)\cong C_0(X)\rtimes^\sigma \Gamma$$ is the full twisted crossed product $C^*$-algebra, and $$C^*_r(X\rtimes \Gamma,\sigma)\cong C_0(X)\rtimes^\sigma_r\Gamma$$ is the reduced twisted crossed product $C^*$-algebra. 
	\end{exmp}

	\section{Quasi-Hermitian groupoids have the weak containment property.}\label{quasi-her-sec}
	
	In this section we prove the main result of the paper, namely that if $L^1(\G,\sigma,\lambda)$ is Hermitian, then $C^*(\G,\sigma,\lambda)=C^*_r(\G,\sigma,\lambda)$. As a consequence, we also prove that if $L^1(\G,\sigma,\lambda)$ is Hermitian, then $L^1(\G,\sigma,\lambda)$
 is spectrally invariant in $C^*_r(\G,\sigma,\lambda)$.	The general strategy is to follow \cite[Section 4]{SaWi}, but not all the steps trivially extend to our situation.

	\begin{defn}
Let $\G$ be a locally compact groupoid with Haar system $\lambda$ and let $\sigma \in Z^2 (\G,\T)$. We say that $\G$  is \emph{$\sigma$-quasi-Hermitian} (resp.\ \emph{$\sigma$-quasi-symmetric}) if $C_c(\G,\sigma)$ is quasi-Hermitian (resp.\ quasi-symmetric) in $L^1(\G,\sigma,\lambda)$. 
	\end{defn}

	\begin{prop}\label{prop:quasi_hermitian_implies_quasi_symmetric} Let $\G$ be a  locally compact   groupoid with Haar system $\lambda$ and $\sigma \in Z^2 (\G,\T)$. If $\G$ is $\sigma$-quasi-symmetric, then $\G$ is $\sigma$-quasi-Hermitian.   
	\end{prop}
\begin{proof} 	The proof of \cite[Proposition 4.1]{SaWi} adapts trivially to give us the following result. Let $f\in C_c(\G,\sigma)_h$, then by assumption $\Spec_{L^1(\G,\sigma)}(f\star_\sigma f^{*_\sigma})\subseteq [0,\infty)$. Therefore
	$$\{\lambda^2:\lambda\in \Spec_{L^1(\G,\sigma)}(f) \}\subseteq \Spec_{L^1(\G,\sigma)}(f\star_\sigma f)=\Spec_{L^1(\G,\sigma)}(f\star_\sigma f^{*_\sigma})\subseteq [0,\infty)\,.$$
	Hence $\Spec_{L^1(\G,\sigma)}(f)\subseteq \R$ for every $f\in C_c(\G,\sigma)_h$.
	\end{proof}

	
	Let $\G$ be a  locally compact groupoid with Haar system $\lambda$, $\sigma \in Z^2 (\G,\T)$ and $1\leq p\leq \infty$. Now fix $x\in \Go$, so we define the representation $L^{\sigma, p}_x:C_c(\G,\sigma)\to B(L^p(\G_x))$  by
	\begin{equation*}
	L^{\sigma,p}_x (f) \xi(\gamma) = \int_{ \G_{x}} \sigma (\gamma\mu^{-1},\mu)f(\gamma\mu^{-1}) \xi(\mu)d\lambda_{x}(\mu),
	\end{equation*}
	for every  $f\in C_c(\G,\sigma)$ and $\gamma \in \G_x$. For $p=2$ this is just \eqref{eq:2-twisted-left-reg-rep}.  
	
	Then we define the \emph{$L^p$-reduced $\sigma$-representation of $\G$} as
	$$L^{\sigma,p}:=\bigoplus_{x\in \Go} L^{\sigma,p}_x:C_c(\G,\sigma)\to \bigoplus_{x\in \Go}B(L^p(\G_x))\,.$$
	
	The following lemma 
	is a straightforward modification of \cite[Lemma 4.6]{ChGaTh} to the situation of general $L^p$-spaces.
	
	\begin{lemma} Let $\G$ be a second-countable locally compact groupoid with Haar system $\lambda$ and  $1\leq p\leq \infty$. Then
\begin{equation}\label{eq_norm}\|f\|_\infty:=\sup_{\gamma\in \G}\{|f(\gamma)|\}\leq\|L^{\sigma,p}(f)\|=\sup_{x\in \Go}\{\|L^{\sigma,p}_x (f)\|_{B(L^p(\G_x))}\}\leq\|f\|_I\,,
	\end{equation}
	for every  $f\in C_c(\G,\sigma)$. 
	\end{lemma}

	\begin{defn} Let $\G$ be a second-countable locally compact   groupoid with Haar system $\lambda$,  $\sigma \in Z^2 (\G,\T)$ and $1\leq p\leq \infty$. The \emph{reduced groupoid $L^p$-Banach algebra}, denoted by $F^p(\G,\sigma,\lambda)$, is the completion of $C_c(\G,\sigma)$ with respect to the norm 
		$$\|f\|_{r,p}:=
		\sup_{x\in \Go}\{\|L^{\sigma,p}_x(f)\|_{B(L^p(\G_x))}\}$$
		for all $f\in C_c(\G,\sigma)$. We will denote $F^p(\G,\sigma,\lambda)$ by $ F^p(\G,\sigma)$ when there is no ambiguity on the Haar system $\lambda$.  
	\end{defn}
	Let  $1\leq p\leq  q\leq \infty$ with $1=\frac{1}{p}+\frac{1}{q}$. Then, given any $x\in \Go$,  there is a duality relation $(L^p(\G_x))^*\cong L^q(\G_x)$ given by 
	$$\langle \xi,\zeta\rangle:=\int_{\G_x}\xi(\gamma) \overline{\zeta(\gamma)}\, d\lambda_{x}\,,$$
	for $\xi\in L^p(\G_x)$ and $\zeta\in L^q(\G_x)$.

	\begin{lemma}\label{dual}
Let $\G$ be a second-countable locally compact   groupoid with Haar system $\lambda$, let $\sigma \in Z^2 (\G,\T)$, and let $1\leq p, q\leq \infty$ be such that  $1=\frac{1}{p}+\frac{1}{q}$. Then 
$$\langle L^{\sigma,p}_x(f^{*_\sigma})\xi,\zeta\rangle =\langle \xi, L^{\sigma,q}_x(f) \zeta\rangle\,,$$
for every $x\in \G^{(0)}$, $f\in L^1(\G,\sigma,\lambda)$, $\xi\in L^p(\G_x,\lambda_x)$ and $\zeta\in L^q(\G_x,\lambda_x)$.
	\end{lemma}
	
	\begin{proof}
Fix  $x\in \G^{(0)}$, $f\in L^1(\G,\sigma)$, $\xi\in L^p(\G_x)$ and $\zeta\in L^q(\G_x)$. Then
	\begin{align*}
	\langle L^{\sigma,p}_x &(f^{*_\sigma})\xi,\zeta\rangle = 
	\int_{\G_x}\left(  \int_{\G_x} \sigma(\gamma\mu^{-1}, \mu)  f^{*_\sigma}(\gamma \mu^{-1}) \xi(\mu)d\lambda_x(\mu)\right)  \overline{\zeta(\gamma)} \,d\lambda_x(\gamma) \\
	&= 
		\int_{\G_x}\left(  \int_{\G_x} \sigma(\gamma\mu^{-1}, \mu)  \overline{\sigma((\gamma \mu^{-1})^{-1}, \gamma \mu^{-1})} \overline{f((\gamma \mu^{-1})^{-1})} \xi(\mu)d\lambda_x(\mu)\right)  \overline{\zeta(\gamma)} \,d\lambda_x(\gamma) \\
	&= \int_{\G_x} \xi(\mu) \left( \int_{\G_x}\overline{  \overline{\sigma(\gamma\mu^{-1}, \mu)} \sigma((\gamma \mu^{-1})^{-1}, \gamma \mu^{-1}) f(\mu \gamma^{-1}) \zeta(\gamma)  }\,d\lambda_x(\gamma)\right)\,d\lambda_x(\mu)  \\
	&= \int_{\G_x} \xi(\mu) \int_{\G_x} \left( \overline{\sigma (\mu \gamma^{-1}, \gamma) f(\mu \gamma^{-1}) \zeta(\gamma)  }\, d\lambda_x(\gamma)\right)\,d\lambda_x(\mu)  \\
	&= \int_{\G_x} \xi(\mu) \overline{L^{\sigma,q}_x (f) \zeta(\mu)}\,d\lambda_x(\mu) = \langle \xi, L^{\sigma,q}_x(f) \zeta\rangle.
	\end{align*} 
	Here we used that $\overline{\sigma(\gamma\mu^{-1}, \mu)} \sigma((\gamma \mu^{-1})^{-1}, \gamma \mu^{-1})  = \sigma (\mu \gamma^{-1}, \gamma) $, or equivalently, that $$\sigma(\mu\gamma^{-1},\gamma\mu^{-1})=\sigma(\mu\gamma^{-1},\gamma)\sigma(\gamma\mu^{-1},\mu)$$ for every $\mu,\gamma \in \G_x$.  
	We get this from \eqref{eq:2-cocycle-associativity-condition} using $\mu \gamma^{-1}$, $(\mu \gamma^{-1})^{-1}$, $\mu$ instead of $\alpha$, $\beta$, $\gamma$, and then applying \eqref{eq:2-cocycle-unit-condition}.
	\end{proof}

	\begin{prop}\label{prop:convolution_algebra}
		Let $\G$ be a second-countable locally compact groupoid with Haar system $\lambda$ and $\sigma \in Z^2 (\G,\T)$, and suppose $1\leq p,q\leq \infty$ with $1=\frac{1}{p}+\frac{1}{q}$. The algebra $L^1(\G,\sigma,\lambda)$ is a normed $*$-algebra with norm 
		$$\|f\|_{\sharp,p}:=\max\{ \|f\|_{r,p},\, \|f\|_{r,q} \}\,,$$
		for $f\in C_c(\G,\sigma)$, with the convolution and involution in $L^1(\G,\sigma,\lambda)$. 
	\end{prop}
	\begin{proof} We will adapt the proof of \cite[Proposition 4.2]{SaWi} to our setting. 
	Given $f\in C_c(\G,\sigma)$ we have that 
			\begin{align*}\|f\|_{\sharp,p} & =\max\{\sup_{x\in \Go}\{ \|L^{\sigma, p}_x(f)\|\},\, \sup_{x\in \Go}\{\|L^{\sigma, q}_x(f)\| \}\} \\ & =\sup_{x\in \Go}\{ \|L^{\sigma, p}_x(f)\|,\, \|L^{\sigma, q}_x(f)\| \}\leq \|f\|_I,\end{align*}
			by (\ref{eq_norm}), so $\|f\|_{\sharp,p}$ is well defined. It is easy to check that $(L^1(\G,\sigma), \|\cdot \|_{\sharp,p})$ is a normed algebra. We only have to prove that the involution is an isometry with respect to $\|\cdot \|_{\sharp,p}$. Let $f\in L^1(\G,\sigma)$. Given $x\in \Go$, by Lemma \ref{dual} we have that 
			\begin{align*}
\|L^{\sigma, p}_x(f^{*_\sigma})\|_{B(L^p(\G_x))} & =\sup\{|\langle L^{\sigma, p}_x(f^{*_\sigma})\xi,\zeta\rangle|: \|\xi\|_p\leq 1,\,\|\zeta\|_q\leq 1 \} \\ 
& =\sup\{|\langle \xi,L^{\sigma, q}_x(f)\zeta\rangle|: \|\xi\|_p\leq 1,\,\|\zeta\|_q\leq 1 \} \\
& = \|L^{\sigma, q}_x(f)\|_{B(L^q(\G_x))}\,.
			\end{align*}
			Similarly, switching $p$ and $q$ we obtain that 
			$$\|L^{\sigma, q}_x(f^{*_\sigma})\|_{B(L^q(\G_x))}=\|L^{\sigma, p}_x(f)\|_{B(L^p(\G_x))}\,,$$
			and therefore $\|f^{*_\sigma}\|_{\sharp,p}=\|f\|_{\sharp,p}$, as desired.
	\end{proof}
	
	\begin{defn}
		Let $\G$ be a second-countable locally compact  groupoid with Haar system $\lambda$ and $\sigma \in Z^2 (\G,\T)$, and let $1\leq p\leq \infty$. The Banach $*$-algebra  $F^p_\sharp(\G,\sigma,\lambda)$  ($F^p_\sharp(\G,\sigma)$ when there is no ambiguity on the Haar system) is defined to be the completion of $L^1(\G,\sigma,\lambda)$ with respect to the norm $\|\cdot \|_{\sharp,p}$.
	\end{defn}
	
	Given $1\leq p,q\leq \infty$ with $1=\frac{1}{p}+\frac{1}{q}$, the Banach $*$-algebras $F^p_\sharp(\G,\sigma,\lambda)$ and $F^q_\sharp(\G,\sigma,\lambda)$ are isometrically isomorphic. Moreover,   $F^2_\sharp(\G,\sigma,\lambda)=C^*_r(\G,\sigma,\lambda)$. 
	
	The following result is a combination of \cite[Theorem 5.1.1]{BergLof} and \cite[Section 10.1]{Cal}.
	\begin{lemma}\label{lemma_inter}
	Let $\G$ be a second-countable locally compact  groupoid with Haar system $\lambda$, let $x\in \Go$, and let $1\leq p_1\leq p_2\leq p_3\leq \infty$. Suppose that $T$ is a bounded operator defined on $L^{p_1}(\G_x,\lambda_x)\cap L^{p_3}(\G_x,\lambda_x)$ such that it extends continuously to bounded operators on both $L^{p_1}(\G_x,\lambda_x)$ and $L^{p_3}(\G_x,\lambda_x)$: Then $T$ extends continuously on $L^{p_2}(\G_x,\lambda_x)$. Furthermore, if $M_i$ is the norm of the extension of $T$ on $L^{p_i}(\G_x,\lambda_x)$ for $i=1,2,3$, then 
	$$M_2\leq M_1^{1-\theta}M_3^\theta\,,$$
	for $0<\theta<1$ satisfying 
	$$\frac{1}{p_2}=\frac{1-\theta}{p_1}+\frac{\theta}{p_3}\,.$$ 
	\end{lemma}

	\begin{prop}\label{prop:interpolation_triple}
		Let $\G$ be a second-countable locally compact   groupoid with Haar system $\lambda$ and $\sigma \in Z^2 (\G,\T)$, and suppose $1< p< 2$. Then 
		$$(L^1(\G,\sigma,\lambda),F^{p}_\sharp(\G,\sigma,\lambda),C^*_{r}(\G,\sigma,\lambda))$$
		is a spectral interpolation triple of reduced 
		Banach $*$-algebras relative to $L^1(\G,\sigma,\lambda)$.
	\end{prop}
	\begin{proof}
Let $q\in (2,\infty)$   such that $1=\frac{1}{p}+\frac{1}{q}$, and let  
$$\theta=\frac{2p-2}{p}\in (0,1)\,,$$
and hence
$$\frac{1}{p}=\frac{1-\theta}{1}+\frac{\theta}{2}\,.$$
Then, for $x\in \Go$ and $f\in L^1(\G, \sigma)$, using Lemma \ref{lemma_inter} and \eqref{eq:I-norm} we have that 
\begin{align*} \|L^{\sigma, p}_x(f)\|_{B(L^{p}(\G_x))} & \leq \|L^{\sigma, 1}_x(f)\|^{1-\theta}_{B(L^{1}(\G_x))}\|L^{\sigma, 2}_x(f)\|^{\theta}_{B(L^{2}(\G_x))} \\
& \leq \|f\|^{1-\theta}_{\sharp,1}\|f\|^{\theta}_{r,2}\leq \|f\|^{1-\theta}_{I}\|f\|^{\theta}_{r,2}\,, 
\end{align*}
and therefore $\|f\|_{r,p}\leq  \|f\|^{1-\theta}_{I}\|f\|^{\theta}_{r,2}$. On the other hand,
\begin{align*}
\frac{1}{q} & =1-\frac{1}{p}=1 - \left( \frac{1- \theta}{1}+\frac{\theta}{2}\right)  \\
& = 1-(1-\theta)-\theta\left( 1-\frac{1}{2}\right)  \\
&= 0+\frac{\theta}{2} = \frac{1-\theta}{\infty}+\frac{\theta}{2}\,,
\end{align*}
so we can apply the same above argument to show that $\|f\|_{r,q}\leq  \|f\|^{1-\theta}_{I}\|f\|^{\theta}_{r,2}$. Hence,
\begin{equation}\label{eq_norms}
\|f\|_{\sharp,p}\leq  \|f\|^{1-\theta}_{I}\|f\|^{\theta}_{r,2}\,,
\end{equation}
for every $f\in L^1(\G, \sigma)$.

Therefore given $f\in L^1(\G, \sigma)$, we have that 
$$\|f^n\|^{\frac{1}{n}}_{\sharp,p}\leq  \|f^n\|^{\frac{1-\theta}{n}}_{I}\|f^n\|^{\frac{\theta}{n}}_{r,2}\,,$$
for every $n\in\N$. Then taking the limit as $n \to \infty$ we have that 
\begin{equation}\label{eq_radius}
r_{F^{p}_\sharp(\G,\sigma)}(f)\leq r_{L^1(\G, \sigma)}(f)^{1-\theta}r_{ C^*_{r}(\G,\sigma)}(f)^\theta\,,
\end{equation}
for every $f\in L^1(\G, \sigma)$. 

To finish the proof we only need to prove that 		$(L^1(\G,\sigma),F^{p}_\sharp(\G,\sigma),C^*_{r}(\G,\sigma))$ is a nested triple of reduced 
Banach $*$-algebras.

Let $\theta\in (0,1)$ be such that 
$$\frac{1}{2}=\frac{1-\theta}{p}+\frac{\theta}{q}\,.$$
Then for $f\in L^1(\G,\sigma)$,
\begin{align*}
\|f\|_{\sharp,p} & = \sup_{x\in \Go}\{\|L^{\sigma, p}_x(f)\|_{B(L^p(\G_x))},\, \|L^{\sigma, q}_x(f)\|_{B(L^q(\G_x))}\} \\
& \geq \sup_{x\in \Go} \{\|L^{\sigma, p}_x(f)\|_{B(L^p(\G_x))}^{1-\theta} \|L^{\sigma, q}_x(f)\|^\theta_{B(L^q(\G_x))}\} \\
& \geq \sup_{x\in \Go} \{\|L^{\sigma, 2}_x(f)\|_{B(L^2(\G_x))}\}=\|f\|_{r,2}\,,
\end{align*} 
by Lemma \ref{lemma_inter}. Therefore the identity map on $L^1(\G,\sigma)$ extends to a contractive $*$-homomorphism 
$$\pi_{p,2}:F^{p}_\sharp(\G,\sigma)\to C^*_r(\G,\sigma)\,.$$
Now we want to see that $\pi_{p,2}$ is injective. Let $f\in \ker(\pi_{p,2})$. Then there exists a sequence $\{f_n\}_{n=1}^\infty$ in $L^1(\G,\sigma)$ such that $\lim f_n=f$ in $F^{p}_\sharp(\G,\sigma)$.  Then given $x\in \Go$ and $\xi\in C_c(\G_x)$ we have that 
$$\lim_{n\to \infty} L^{\sigma, 2}_x(f_n)\xi=0\qquad \text{in }L^2(\G_x)\,,$$
that is
\begin{equation*}
\begin{split}
0 &= \| \lim_{n\to \infty} L^{\sigma, 2}_x(f_n)\xi\|_{L^2(\G_x)} \\
&= \left( \int_{\G_x} \lim_{n \to \infty} \vert L^{\sigma, 2}_x (f_n)\xi(\gamma)\vert^2 \,d\lambda_x(\gamma)\right)^{1/2} \\
&=  \left(\int_{\G_x} \lim_{n \to \infty} \left|  \int_{\G_x} \sigma(\gamma \mu^{-1}, \mu) f_n(\gamma \mu^{-1}) \xi(\mu) \,d\lambda_x(\mu) \right|^2 \,d\lambda_x(\gamma)   \right)^{1/2} \,,
\end{split}
\end{equation*}
which forces
\begin{equation}\label{vanish_int}
	\lim_{n\to \infty} \left| \int_{\G_x} \sigma(\gamma \mu^{-1}, \mu) f(\gamma \mu^{-1}) \xi(\mu) \,d\lambda_x(\mu) \right| = 0\,.
\end{equation}
Now observe that the map 
$$\Psi:L^1(\G,\sigma)\to B\left( \oplus_{x\in \Go} \left( L^p(\G_x)\oplus L^q(\G_x)\right) \right) \,,$$
defined by $f\mapsto  \bigoplus_{x\in \Go} (L^{\sigma, p}_x(f) \oplus  L^{\sigma, q}_x(f))$, extends isometrically to a map   $$\Psi:F^{p}_\sharp(\G,\sigma)\to B\left( \oplus_{x\in \Go} \left( L^p(\G_x)\oplus L^q(\G_x)\right) \right)\,.$$
Now fix $x\in \Go$ and  $i\in\{p,q\}$. Then given $\xi\in C_c(\G_x)$ we have that 
$$\Psi(f)\xi=\lim_{n\to\infty}\Psi(f_n)\xi\,,$$
but
\begin{equation*}
	\begin{split}
	\Vert \Psi(f)\xi \Vert_{L^i (\G_x)} &= \|\lim_{n\to\infty} \Psi(f_n)\xi\|_{L^i({\G_x})}\\
	 &=\| \lim_{n\to\infty} L^{\sigma, i}_x(f_n)\xi\|_{L^i({\G_x})}\\
	&=\lim_{n\to \infty} \left(\int_{\G_x} \left|  \int_{\G_x} \sigma(\gamma \mu^{-1}, \mu) f_n(\gamma \mu^{-1}) \xi(\mu) \,d\lambda_x(\mu) \right|^i \,d\lambda_x(\gamma)   \right)^{1/i} \\
	&= 0
	\end{split}
\end{equation*}
because of (\ref{vanish_int}). Thus, $\Psi(f)=0$ and since $\Psi$ is isometric we have that $f=0$. Hence $\pi_{p,2}$ is injective.

Now, by (\ref{eq_norms}) and the fact that the regular representation is $I$-norm bounded, we get
$$\|f\|_{\sharp,p}\leq \|f\|^{1-\theta}_{I}\|f\|^{\theta}_{r,2}\leq \|f\|^{1-\theta}_{I}\|f\|^{\theta}_{I}=\|f\|_{I}\,,$$
for every $f\in L^1(\G,\sigma)$. Hence, the identity map on $L^1(\G,\sigma)$ extends to a contraction 
$$\Phi_{p,I}:L^1(\G,\sigma) \to F^{p}_\sharp(\G,\sigma)\,.$$
Observe that then  $(\pi_{p,2}\circ \Phi_{p,I}):L^1(\G,\sigma)\to C^*_r(\G,\sigma)$ is the regular representation of $L^1(\G,\sigma)$, which is injective. It follows that $\Phi_{p,I}:L^1(\G,\sigma)\to F^{p}_\sharp(\G,\sigma)$ is injective too. 
	\end{proof}

	\begin{prop}\label{prop:equivalence_quasi_symmetric_hermitian}
		Let $\G$ be a second-countable locally compact  groupoid with Haar system $\lambda$ and $\sigma \in Z^2 (\G,\T)$. Then the following statements are equivalent:
		\begin{enumerate}
			\item $\G$ is $\sigma$-quasi-symmetric,
			\item $\G$ is $\sigma$-quasi-Hermitian,
			\item $r_{L^1(\G,\sigma,\lambda)}(f)=r_{C^*_r(\G,\sigma,\lambda)}(f)$ for every $f\in C_c(\G,\sigma)_h$,
			\item $\Spec_{L^1(\G,\sigma,\lambda)}(f)=\Spec_{C^*_r(\G,\sigma,\lambda)}(f)$ for every  $f\in C_c(\G,\sigma)$.
		\end{enumerate}	
	\end{prop}
	\begin{proof}
	 $(1)\Rightarrow (2)$ was proved in Proposition \ref{prop:quasi_hermitian_implies_quasi_symmetric}. $(3)\Rightarrow (4)$ is proved in Theorem \ref{Hulanicki}, and $(4)\Rightarrow (1)$  is clear. So we only need to prove $(2)\Rightarrow (3)$. 
	
	Suppose that $\G$ is $\sigma$-quasi-Hermitian and $1<p<2$. Then 
	$$(L^1(\G,\sigma),F^{p}_\sharp(\G,\sigma),C^*_{r}(\G,\sigma))$$
	is a spectral interpolation triple of reduced Banach $*$-algebras 
	relative to  $L^1(\G,\sigma)$ by Proposition \ref{prop:interpolation_triple}. Hence, by Proposition \ref{sp_int_triple} we have that 
	$$r_{F^{p}_\sharp(\G,\sigma)}(f)=r_{C^*_{r}(\G,\sigma)}(f)\,,$$
	for every $f\in C_c(\G,\sigma)_h$.
	
	Now fix $f\in C_c(\G,\sigma)_h$. Then the sets $U=\text{Supp}(f)$ and $s(U)$ are compact sets. Replacing $U$ by $U\cup U^{-1}$ we can assume that $r(U)=s(U)$, and that given any $x\in s(U)$, the map $\G_x\cap U\to \G^x\cap U$ given by $\gamma\mapsto \gamma^{-1}$ is a bijection. Then  since $\G$ is locally compact, using a partition of the unit, there exists a function $g_1\in C_c(\G,\sigma)$ such that $g_1(\gamma)=1$ for every $\gamma\in U$. Then since the map $x\mapsto \int_{\G_x}g_1(\gamma)\,d\lambda_x(\gamma)$ is continuous we have that 
\begin{equation}\label{bounded}
K:=\sup\left\lbrace \int_{\G_x}g_1(\gamma)\,d\lambda_x(\gamma):x\in s(U) \right\rbrace <\infty\,.
\end{equation} 
	Observe that given $x\in s(U)$ we have that 
	$$\int_{\G_x}g_1(\gamma)\,d\lambda_x(\gamma)=\lambda_x(\G_x\cap U)\leq K\,.$$
	
	Now given $n\in\N$, we denote by $f^n$ the $n$'th convolution power $f\star_\sigma\cdots \star_\sigma f$. Then we have that 
	$$\text{Supp}(f^n)\subseteq U^{(n)}=\{ \gamma_1\cdots \gamma_n: \gamma_i\in U\text{ such that }r(\gamma_i)=s(\gamma_{i+1})\text{ for }i=1,\dots,n-1  \}\,.$$
	By continuity of the groupoid product $U^{(n)}$ is a compact subset of $\G$. Now given $x\in s(U^{(n)})$ we define $U^{(n)}_x:=U^{(n)}\cap \G_x$, so 
	\begin{align*}
\lambda_x(U^{(n)}_x) & =\int_{U^{(n-1)}_x} \lambda_{s(\gamma)}(U_{r(\gamma)}\gamma)\,d\lambda_{x}(\gamma) \leq  \int_{U^{(n-1)}_x} \lambda_{r(\gamma)}(U_{r(\gamma)})\,d\lambda_{x}(\gamma) \\
& \leq  \int_{U^{(n-1)}_x} K \,d\lambda_{x}(\gamma) = K\lambda_x(U^{(n-1)}_x) \\
& \leq K^2 \lambda_x(U^{(n-2)}_x)\leq \cdots \leq K^{n-1} \lambda_x(\G_x\cap U) \\
& \leq K^n\,,
    \end{align*}
	by using the invariance of the Haar  measures.

	 Now, fix  $n\in \N$ and $x\in s(U^{(n)})$.  Let $1_{U^{(n)}_x}\in L^1(\G_x)$ be the characteristic function on $U^{(n)}_x$. 
	Then we have that 
	\begin{align*}
	\|f^n_{|\G_x}\|_{L^1(\G_x)} & = \| f^n_{|\G_x} 1_{U^{(n)}_x} \|_{L^1(\G_x)}  \\
	& \leq \| f^n_{|\G_x}  \|_{L^p(\G_x)}\|1_{U^{(n)}_x} \|_{L^q(\G_x)} \qquad (\text{where }1=\frac{1}{p}+\frac{1}{q}) \\
	& \leq \|L^{p,\sigma}_x(f^{n-1})\|_{B(L^p(\G_x))} \|f_{|\G_x}\|_{L^p(\G_x)} K^{\frac{n}{q}} \\
	& \leq \|f^{n-1}\|_{\sharp, p}\|f_{|\G_x}\|_{L^p(\G_x)} K^{\frac{n}{q}}\,.
	\end{align*}  
In a similar way we have that 	
$$\|(f^{*_\sigma})^n_{|\G_x}\|_{L^1(\G_x)}\leq\|(f^{*_\sigma})^{n-1}\|_{\sharp, p}\|f^{*_\sigma}_{|\G_x}\|_{L^p(\G_x)} K^{\frac{n}{q}}=\|f^{n-1}\|_{\sharp, p}\|f^{*_\sigma}_{|\G_x}\|_{L^p(\G_x)} K^{\frac{n}{q}}\,.$$ 
Now using (\ref{bounded}) but replacing  $g_1$ with $|f|^p$ and $|f^{*_\sigma}|^p$, we have that
$$\sup_{x\in \Go}\{\max\{\|f_{|\G_x}\|_{L^p(\G_x)},\|f^{*_\sigma}_{|\G_x}\|_{L^p(\G_x)}\}=C<\infty\,,$$
and then 
\begin{align*}
\|f^n\|_{I} & =\sup_{x \in \Go} \{\max\{\|f^n_{|\G_x}\|_{L^1(\G_x)},\|(f^{*_\sigma})^n_{|\G_x}\|_{L^1(\G_x)} \}\} \\
& \leq \|f^{n-1}\|_{\sharp, p} \sup_{x \in \Go}\{ \max\{\|f_{|\G_x}\|_{L^p(\G_x)},\|(f^{*_\sigma})_{|\G_x}\|_{L^p(\G_x)} \}\} K^{\frac{n}{q}} \\
& \leq  \|f^{n-1}\|_{\sharp, p} C K^{\frac{n}{q}}\,.
\end{align*}
Therefore,
	$$\|f^n\|^{\frac{1}{n}}_{I}\leq   \|f^{n-1}\|^{\frac{1}{n}}_{\sharp, p} C^{\frac{1}{n}} K^{\frac{1}{q}}\,,$$
and then when $n\to \infty$, we have that
		$$r_{L^1(\G,\sigma)}(f) \leq  r_{F^{p}_\sharp(\G,\sigma)}(f)K^{\frac{1}{q}}=r_{C^*_r(\G,\sigma)}(f)K^{\frac{1}{q}}\,.$$
Then, taking the limit for $p\to 1^+$, we have that $q\to \infty$, so $r_{L^1(\G,\sigma)}(f) \leq  r_{C^*_r(\G,\sigma)}(f)$. But we always have that $r_{C^*_r(\G,\sigma)}(f)\leq r_{L^1(\G,\sigma)}(f)$, and hence 
	$$r_{C^*_r(\G,\sigma)}(f)= r_{L^1(\G,\sigma)}(f)\,.$$

	\end{proof}

	\begin{thm}\label{corl:quasi_symmetric_approx_theory}
		Let $\G$ be a second-countable locally compact   groupoid with Haar system $\lambda$ and $\sigma \in Z^2 (\G,\T)$. If $\G$ is $\sigma$-quasi-Hermitian, then $C_r^*(\G,\sigma)$ is the $C^*$-envelope of $L^1(\G,\sigma)$. In particular, $\G$ with Haar system $\lambda$ and the twist $\sigma$ has the weak containment property.  
	\end{thm}
	\begin{proof}
	By 	Proposition \ref{prop:equivalence_quasi_symmetric_hermitian}, if $C_c(\G,\sigma)$ is quasi-Hermitian in $L^1(\G,\sigma,\lambda)$, then for every $f\in C_c(\G,\sigma)_h$ we have $r_{C^*_r(\G,\sigma)}(f)= r_{L^1(\G,\sigma)}(f)$. Therefore by Proposition \ref{same_C_env} we have that $C^*_r(\G,\sigma,\lambda)$ is the $C^*$-envelope of $L^1(\G,\sigma,\lambda)$. But this means that the reduced norm is the maximal norm, and  so $C^*_r(\G,\sigma, \lambda)=C^*(\G,\sigma, \lambda)$.
	\end{proof}

  
 
 Finally, we address the problem of spectral invariance. Recall that spectra in non-unital algebras are defined in terms of the spectra in their minimal unitizations. Since a Banach $*$-algebra $\A$ is Hermitian if and only if its minimal unitization is Hermitian \cite[Theorem (4.7.9)]{Rickart60}, we obtain the following corollary.
 
 \begin{corl}\label{Wiener_env} Let $\G$ be  a second-countable locally compact groupoid with Haar system $\lambda$ and $\sigma\in Z^2 (\G,\T)$. Then $L^1(\G,\sigma,\lambda)$ is Hermitian if and only if $L^1(\G,\sigma,\lambda)$ is spectrally invariant in $C_r^*(\G,\sigma,\lambda)$. 
 \end{corl}
 \begin{proof}
 	Suppose that $L^1(\G,\sigma)$ is Hermitian. Then $C^*(L^1(\G,\sigma))=C^*_r(\G,\sigma)$ by Theorem \ref{corl:quasi_symmetric_approx_theory}. But as $L^1 (\G,\sigma)$ is Hermitian, it must be spectrally invariant in its enveloping $C^*$-algebra, hence it is spectrally invariant in $C^*_r(\G,\sigma)$.
 	The converse implication is trivial. 
 	
 \end{proof}

 \section{Non-Hermitian ample amenable groupoids}\label{sec_non_herm}
 In the previous section we proved that if $\G$ is a second-countable locally compact quasi-Hermitian groupoid, then $\G$ satisfies the weak containment property, i.e.\ $C^*_r (\G) \cong C^* (\G)$. 
 In the case $\G$ is a 
 group, this translates to the fact that quasi-Hermitian groups are amenable. In the case of groupoids the situation is more subtle, since the weak containment property is not equivalent to $\G$ being amenable. There are examples of non-amenable groupoids with the weak containment property in \cite{AlFi,Will}. In this section we will see that when $\G$ is an ample groupoid, an obstruction to $\G$ being quasi-Hermitian is that the isotropy groups are not quasi-Hermitian. We can then easily construct an amenable ample groupoid that is not Hermitian. 
 
 

 \begin{defn}
 	A locally compact Hausdorff étale groupoid is called \emph{ample} if it has a basis consisting of open and compact bisections.
 \end{defn}

 Let $\G$ be  an ample groupoid and let $\sigma\in Z^2 (\G,\T)$. 
 Then, given $x\in \Go$, the restriction of $\sigma$ to the isotropy group $\G_x^x$ is a group $2$-cocycle. We denote this restricted $2$-cocycle by $\sigma_x$.

 \begin{prop}\label{non-herm}
 	Let $\G$ be a second-countable ample groupoid  and  $\sigma\in Z^2 (\G,\T)$, with $\Go$ compact. Suppose that for every $\gamma\in \Iso(\G)$ there exists a clopen bisection $U\subseteq \G$ such that $\gamma\in U$ and $r(U)=s(U)=\Go$. If $\G$ is $\sigma$-quasi-Hermitian, then $\G^x_x$ is a $\sigma_x$-quasi-Hermitian group for every $x\in \Go$.
 \end{prop}
 \begin{proof} Let  us suppose that there exists $x\in \Go$ such that $\G^x_x$ is not $\sigma_x$-quasi-Hermitian. Then there exists $f\in C_c(\G^x_x,\sigma_x)_h$ such that $\Spec_{\ell^1(\G^x_x,\sigma_x)}(f)\nsubseteq \R$. Observe that $f$ must be of the form 
 	$$f=\sum_{i=1}^m \lambda_i \delta_{\gamma_i}+\sum_{i=1}^m \overline{\lambda_i\sigma(\gamma_i^{-1},\gamma_i)} \delta_{\gamma^{-1}_i}\,,$$
where $\lambda_i\in \C$  and $\gamma_i\in \G^x_x$.	
 	 Let $\lambda \in \Spec_{\ell^1(\G^x_x,\sigma_x)}(f)\setminus \R$. By assumption,  for every $\gamma\in \G_x^x$ there exists a bisection $U_\gamma\subseteq \G$ such that $\gamma\in U_\gamma$ and $r(U_\gamma)=s(U_\gamma)=\Go$. Then $\hat{f}=\sum_{i=1}^m\lambda_i1_{U_{\gamma_i}}+\sum_{i=1}^m \overline{\lambda_i\sigma(\gamma_i^{-1},\gamma_i)}1_{U^{-1}_{\gamma_i}}\in C_c(\G,\sigma)_h$. We claim that $\lambda\in \Spec_{\ell^1(\G,\sigma)}(\hat{f})$, that is $\lambda 1_{\Go} - \hat{f}$ is not invertible in $\ell^1(\G,\sigma)$.  Suppose that $\lambda 1_{\Go} - \hat{f}$ is  invertible in $\ell^1(\G,\sigma)$, so there exists $\hat{g}\in \ell^1(\G,\sigma)$ such that $(\lambda 1_{\Go} - \hat{f})\star_\sigma \hat{g}=1_{\Go}$. Let $\{\hat{g}_n\}_{n=1}^\infty$  be a sequence in $C_c(\G,\sigma)$ such that $\hat{g}_n\to \hat{g}$ in $\ell^1(\G,\sigma)$, and hence $(\lambda 1_{\Go} - \hat{f})\star_\sigma \hat{g}_n\to 1_{\Go}$ in $\ell^1(\G,\sigma)$. In particular, $((\lambda 1_{\Go} - \hat{f})\star_\sigma \hat{g}_n)(x)\to 1$ and $((\lambda 1_{\Go} - \hat{f})\star_\sigma \hat{g}_n)(\gamma)\to 0$ 
 for every $\gamma\in \G_x^x\setminus \{x\}$. Let $\hat{g}_n=\sum_{j=1}^{l_n} \beta_{j,n} 1_{V_{j,n}}$ where the $V_{j,n}$'s are compact open bisections. Then  
 	\begin{align*}
 	(\lambda 1_{\Go} - \hat{f})\star_\sigma \hat{g}_n & =(\lambda 1_{\Go} - \hat{f})\star_\sigma (\sum_{j=1}^{l_n} \beta_{j,n} 1_{V_{j,n}}) \\
 	& = \sum_{j=1}^{l_n} \lambda \beta_{j,n} 1_{V_{j,n}} - \sum_{j=1}^{l_n} \beta_{j,n} (\hat{f}\star_\sigma  1_{V_{j,n}})\,. \\
 	\end{align*}
 	Therefore, defining $\eta_{j,n}:=xV_{j,n}x\in \G^x_x$  and $g_n=\sum_{j=1}^{l_n} \beta_{j,n} \delta_{\eta_{j,n}}\in \ell^1(\G^x_x,\sigma_x)$, we have that the sequence $\{g_n\}_{n=1}^\infty$ converges in $\ell^1(\G^x_x,\sigma_x)$ because $\{\hat{g}_n\}_{n=1}^\infty$ converges in $\ell^1(\G,\sigma)$, and
 	$$(\lambda1 -f )\star_{\sigma_x} g_n\to (\lambda1 -f )\star_{\sigma_x} g=1\,.$$ 
 	Therefore $\lambda\notin \Spec_{\ell^1(\G^x_x,\sigma_x)} (f)$, a contradiction.		
 \end{proof}

 \begin{rmk}
 It was observed in \cite[Lemma 4.9]{NyOrt} that the condition that for every $\gamma\in \Iso(\G)$ there exists a clopen bisection $U\subset \G$ such that $\gamma\in U$ and $r(U)=s(U)=\Go$ is satisfied if $|\text{Orb}_\G(x)|\geq 2$ for every $x\in \Go$.
 
 \end{rmk}
 \begin{exmp}
 	Willett  constructed in \cite{Will} a second-countable locally compact ample groupoid $\G$ with $\Go$ compact that satisfies the weak-containment property. $\G$ is a group bundle so it clearly satisfies the assumptions in Proposition \ref{non-herm}. Moreover, $\G$ has an isotropy group isomorphic to the free non-abelian group with two generators, which is not quasi-Hermitian by \cite[Corollary 4.8]{SaWi}. 
 	Therefore, by Proposition \ref{non-herm} we have that $\G$ is not quasi-Hermitian. 
 \end{exmp}
 
 \begin{exmp}
 	Let $\Gamma$ be a countable discrete group with unit $e$, and let us consider an action of $\Gamma$ on a second-countable compact Hausdorff space $X$. 
 	Then $X\rtimes \Gamma$ is a second-countable locally compact Hausdorff étale groupoid.
 	Let us suppose that $\Gamma$ contains a free semigroup on two generators $z,t$. Given $\gamma\in \Gamma$ we define the bisection  $U_\gamma=(X,\gamma)$ of  $X\rtimes \Gamma$. Observe that given $\gamma,\gamma'\in \Gamma$ we have that $U_\gamma U_{\gamma'}=U_{\gamma\gamma'}$.  Let us consider $$f=a_0 1_{U_e}+a_1 1_{U_z}+a_21_{U_{z^2}}\in C_c(X\rtimes \Gamma)\,,$$
 	where $a_0,a_1,a_2\in \C$ satisfy  $|a_0|=|a_1|=|a_2|=\frac{1}{3}$ and 
 	$$\sup\{|a_0+a_1x+a_2x^2|:x\in \mathbb{T} \}<1\,.$$
 	Observe that $1_{U_e}$ is the unit of $\ell^1(X\rtimes \Gamma)$. We then have
 	$$r_{\ell^1(X\rtimes \Gamma)}(f)<1\,,$$
 	by using the spectral mapping theorem and maximum modulus principle,
	and since $f$ is normal, i.e. $ff^*=f^*f$, we obtain
 	$$\|f\|_{C^*_r(X\rtimes \Gamma)}=r_{C^*_r(X\rtimes \Gamma)}(f)<1\,.$$
 	Now, since $1_{U_t}$ is a unitary in $C^*(X\rtimes \Gamma)$ it follows that
 	$$\|f\star_\sigma 1_{U_t}\|=\|a_0 1_{U_t}+a_11_{U_{zt}}+a_21_{U_{z^2t}}\|\,.$$
 	Then since $t,z$ generate a free non-abelian group we have that 
 	$$(a_0 1_{U_t}+a_11_{U_{zt}}+a_21_{U_{z^2t}})^n$$
 	has $3^n$ linearly independent terms. Hence, since $U_\gamma\cap U_{\gamma'}=\emptyset$ if and only if $\gamma\neq \gamma'$, we have that
 	\begin{align*}
 	\|(a_0 1_{U_t}+a_11_{U_{zt}}+a_21_{U_{z^2t}})^n\|_{\ell^1(X\rtimes \Gamma)} &=1
 	\end{align*}
 	for every $n\in \N$, and so $r_{\ell^1(X \rtimes \Gamma)}(a_0 1_{U_t}+a_11_{U_{zt}}+a_21_{U_{z^2t}})=1$. Thus we have that 
 	$$\Spec_{\ell^1(X\rtimes \Gamma)}(a_0 1_{U_t}+a_11_{U_{zt}}+a_21_{U_{z^2t}})\neq \Spec_{C^*_r(X\rtimes \Gamma)}(a_0 1_{U_t}+a_11_{U_{zt}}+a_21_{U_{z^2t}})\,,$$
 	and hence by Proposition \ref{prop:equivalence_quasi_symmetric_hermitian} we have that $X\rtimes \Gamma$ is not quasi-Hermitian. 
 	
 	Then, let $\mathbb{F}_2$ be the non-abelian free group with two generators. It is known that there exists a locally compact space $X$ and an amenable action of $\mathbb{F}_2$ on $X$ (see for example \cite{Suz}). Hence, $X\rtimes \mathbb{F}_2$ is an amenable groupoid but not quasi-Hermitian.  
 	
 \end{exmp}

\section{Hermitian twisted groupoid Banach $*$-algebras}\label{sec_exa_herm_twist}
Let $\G$ be a second-countable locally compact groupoid with Haar system $\lambda$ and let $\sigma\in Z^2 (\G,\T)$. 
In this section we give a sufficient condition for the Banach $*$-algebra $L^1 (\G, \sigma,\lambda)$ to be Hermitian. As a consequence we are able to give conditions for twisted transformation groupoids so that the associated twisted transformation groupoid Banach $*$-algebras are Hermitian.
	
	 \begin{defn}
		Given  a second-countable locally compact groupoid $\G$ with Haar system $\lambda$ and $\sigma\in Z^2 (\G,\T)$, we define the \emph{twisted groupoid} $\G_\sigma$ to be the groupoid $\G\times \mathbb{T}$ with product topology and operations defined by
		$$(\gamma_1,z_1)\cdot (\gamma_2,z_2)=(\gamma_1\gamma_2,z_1z_2\overline{\sigma(\gamma_1,\gamma_2)}) \qquad \text{if }(\gamma_1, \gamma_2 )\in \G^{(2)}\,,$$
		and 
		$$(\gamma,z)^{-1}=(\gamma^{-1},\overline{z}\sigma(\gamma,\gamma^{-1}))\,.$$
		The Haar system of $\G_\sigma$ is the one given by $\lambda\times\eta=\{\lambda_x\times \eta\}_{x\in \Go}$, where $\eta$ is the normalized Lebesgue measure on $\mathbb{T}$.    
		
	\end{defn}

	\begin{prop}\label{prop_inc_iso}
		Let $\G$ be  a second-countable locally compact groupoid with Haar system $\lambda$ and $\sigma\in Z^2 (\G,\T)$. The  map 
		$$j:L^1(\G,\sigma,\lambda)\to L^1(\G_\sigma,\lambda\times\eta)\,,$$
		given by $j(f)(\gamma,z)=zf(\gamma)$ is an isometric $*$-homomorphism.
	\end{prop}
	\begin{proof}
		First we prove that $j$ is a $*$-homomorphism. Fix $f,g\in L^1(\G,\sigma)$, $\gamma\in \G$ and $z\in \mathbb{T}$. Then we have that 
		\begin{align*}
		j(f\star_\sigma g)(\gamma,z) & = z(f\star_\sigma g)(\gamma)=z\int_{\G_{s(\gamma)}} \sigma(\gamma\mu^{-1},\mu)f(\gamma\mu^{-1})g(\mu)\, d\lambda_{s(\gamma)}(\mu) \\
		& =\int_{\G_{s(\gamma)}} z\overline{\sigma(\gamma,\mu^{-1})}\sigma(\mu,\mu^{-1}) f(\gamma\mu^{-1}) g(\mu)\,d\lambda_{s(\gamma)} (\mu) dt  \\
		& =\int_{\mathbb{T}}\int_{\G_{s(\gamma)}} z\overline{\sigma(\gamma,\mu^{-1})}\sigma(\mu,\mu^{-1}) f(\gamma\mu^{-1}) g(\mu)\,d\lambda_{s(\gamma)} (\mu) dt \\
		& =\int_{\mathbb{T}}\int_{\G_{s(\gamma)}} \overline{t\sigma(\gamma,\mu^{-1})}z  \sigma(\mu,\mu^{-1}) f(\gamma\mu^{-1}) t g(\mu)\,d\lambda_{s(\gamma)} (\mu) dt\\
		& =\int_{\mathbb{T}}\int_{\G_{s(\gamma)}} j(f)(\gamma\mu^{-1}, \overline{t\sigma(\gamma,\mu^{-1})}z\sigma(\mu,\mu^{-1})) j(g)(\mu,t)\,d\lambda_{s(\gamma)} (\mu) dt \\
		& =\int_{\mathbb{T}}\int_{\G_{s(\gamma)}} j(f) ((\gamma,z)(\mu,t)^{-1}) j(g)(\mu,t)\,d\lambda_{s(\gamma)} (\mu) dt \\
		& =(j(f)\star j(g))(\gamma,z)\,,
		\end{align*}  	 
		where, at the third equality, we have used $\sigma(\gamma \mu^{-1}, \mu)\sigma(\gamma, \mu^{-1}) = \sigma(\mu^{-1}, \mu)$. This identity follows from \eqref{eq:2-cocycle-associativity-condition} by using $\gamma$, $\mu^{-1}$, $\mu$ instead of $\alpha$, $\beta$, $\gamma$ and then applying \eqref{eq:2-cocycle-unit-condition}. We then use that $\sigma(\mu^{-1}, \mu) = \sigma(\mu, \mu^{-1})$, which follows from \eqref{eq:2-cocycle-associativity-condition} using $\mu$, $\mu^{-1}$, $\mu$ instead of $\alpha$, $\beta$, $\gamma$ and then applying \eqref{eq:2-cocycle-unit-condition}.
		
		Moreover, 
		\begin{align*}
		j(f^{*_\sigma})(\gamma,z) & =zf^{*_\sigma}(\gamma)=z\overline{\sigma(\gamma,\gamma^{-1}) f(\gamma^{-1})} \\
		& = \overline{j(f)(\gamma^{-1},\overline{z}\sigma(\gamma,\gamma^{-1}))}\\
		& = \overline{j(f)((\gamma, z)^{-1})} \\
		& = j(f)^* (\gamma, z) \,.
		\end{align*}
		Finally,
		\begin{align*}
		\int_{\G_x} |f(\gamma)|\, d\lambda_x(\gamma) & =  \int_{\mathbb{T}}|z|\int_{\G_x} |f(\gamma)|\, d\lambda_x(\gamma)\,dz \\
		& =  \int_{\mathbb{T}}\int_{\G_x} |zf(\gamma)|\,d\lambda_x(\gamma)\,dz \\
		& =  \int_{\mathbb{T}}\int_{\G_x} |j(f)(\gamma,z)|\, d\lambda_x(\gamma)\,dz
		\end{align*}
		and 
		\begin{align*}
		\int_{\G_x} |f(\gamma^{-1})|\, d\lambda_x(\gamma) & =  \int_{\mathbb{T}}|\overline{z}|\int_{\G_x} |f(\gamma^{-1})|\, d\lambda_x(\gamma)\,dz \\
		& =  \int_{\mathbb{T}}\int_{\G_x} |\overline{z}\sigma(\gamma,\gamma^{-1})f(\gamma^{-1})|\,d\lambda_x(\gamma)\,dz \\
		& =  \int_{\mathbb{T}}\int_{\G_x} |j(f)(\gamma^{-1},\overline{z}\sigma(\gamma,\gamma^{-1}))|\, d\lambda_x(\gamma)\,dz \\
		& =  \int_{\mathbb{T}}\int_{\G_x} |j(f)((\gamma,z)^{-1}))|\, d\lambda_x(\gamma)\,dz\,.
		\end{align*}
	 Therefore, 
		\begin{align*}\|f\|_I & =\sup_{x \in \Go} \max \left\lbrace\int_{\G_x} |f(\gamma)|\, d\lambda_x(\gamma)\,,\int_{\G_x} |f(\gamma^{-1})|\, d\lambda_x(\gamma)\right\rbrace    \\
		& =\sup_{x \in \Go} \max\left\lbrace \int_{\mathbb{T}}\int_{\G_x} |j(f)(\gamma,z)|\, d\lambda_x(\gamma)\,dz\,,\int_{\mathbb{T}}\int_{\G_x} |j(f)((\gamma,z)^{-1})|\, d\lambda_x(\gamma)\,dz\right\rbrace \\
		&=\|j(f)\|_I\,.
		\end{align*}	
	\end{proof}

	\begin{prop}\label{thm_herm}
		Let $\G$ be  a second-countable locally compact groupoid with Haar system $\lambda$ and $\sigma\in Z^2 (\G,\T)$. Suppose that $L^1(\G_\sigma)$ is Hermitian. Then $L^1(\G,\sigma)$ is Hermitian.
	\end{prop}
	\begin{proof}
		By Proposition \ref{prop_inc_iso}  $L^1(\G,\sigma)$ is a closed Banach $*$-subalgebra of $L^1(\G_\sigma)$. Then by \cite[Proposition 7.10]{Bill} $L^1(\G,\sigma)$ is Hermitian. 
	\end{proof}

	\begin{exmp}\label{exampl_trans}
	Let $X$ be a second-countable locally compact Hausdorff space, and let $\Gamma$ be a second-countable locally compact group acting on $X$ by homeomorphisms. Let $\G=X\rtimes \Gamma$ be the transformation groupoid, which is locally compact.  Recall that $\G$ is étale if and only if $\Gamma$ is discrete. Let $\sigma\in Z^2 (\Gamma,\T)$, and extend $\sigma$ to a $2$-cocycle of $X\rtimes \Gamma$ as shown in Example \ref{example_trans}.
		We define $\Gamma_\sigma$ to be the group that is $\Gamma\times \mathbb{T}$ with the product topology, and  
	$$(\gamma_1,z_1)(\gamma_2,z_2)=(\gamma_1\gamma_2,z_1z_2\overline{\sigma(\gamma_1,\gamma_2)})\,,$$
	 	for every $z_1,z_2\in \mathbb{T}$ and $\gamma_1,\gamma_2\in \Gamma$.   
	 	Now if we define the action of $\Gamma_\sigma$ on $X$ by 
	 	$$(\gamma,z)\cdot x:=\gamma\cdot x\,,$$
	 	then the transformation groupoid $X\rtimes \Gamma_\sigma$ is isomorphic to $(X\rtimes \Gamma)_\sigma$.

	\end{exmp}

	Let $X$ be a second-countable locally compact Hausdorff space, and let $\Gamma$ be a second-countable locally compact  group with modular function $\Delta$, acting on $X$ by homeomorphisms. Further, let $\sigma\in Z^2 (\Gamma,\T)$, and let $C_0(X)$ be the Banach $*$-algebra of continuous functions on $X$ that vanish at infinity with the supremum norm $\|\cdot \|_\infty$. Then $\Gamma$ acts on $C_0(X)$ by $\gamma\cdot f(x)=f(\gamma^{-1}\cdot x)$ for every $f\in C_0(X)$ and $\gamma\in \Gamma$. Let us define the \emph{generalized $L^1$-algebra $L^1(\Gamma,C_0(X),\sigma)$} to be the completion of 
	$$C_c(\Gamma,C_0(X),\sigma)=\{f:\Gamma\to C_0(X): \text{$f$ is continuous with compact support}\}$$
	with respect to the norm 
	$$\|f\|:=\int_\Gamma \|f(\gamma)\|_\infty \,d\lambda(\gamma)\,,$$
	where $\lambda$ is a left Haar measure of $\Gamma$.
	Then $L^1(\Gamma,C_0(X),\sigma)$ becomes a Banach $*$-algebra with  the operations 
		 $$(f\star_\sigma g)(\gamma)(x)=\int_\Gamma \sigma(\gamma\mu^{-1},\mu) f(\gamma\mu^{-1})(\mu\cdot x)g(\mu)(x)\,d\lambda(\mu)\,,$$
		 and 
		 $$(f^{*_\sigma})(\gamma)(x)=\Delta(\gamma^{-1})\overline{\sigma(\gamma,\gamma^{-1} )f(\gamma^{-1})(\gamma\cdot x)}\,,$$
		 for every $f,g\in L^1(\Gamma,C_0(X),\sigma)$, $\gamma\in \Gamma$ and $x\in X$. 
		 
		If $\sigma$ is the trivial twist, we denote  $L^1(\Gamma,C_0(X),\sigma)$ by $L^1(\Gamma,C_0(X))$. The $C^*$-envelope of $L^1(\Gamma,C_0(X),\sigma)$ is the twisted crossed product $C^*$-algebra $C_0(X)\rtimes^\sigma \Gamma$.

		 \begin{lemma}\label{lemma_iso} Let $X$ be a second-countable locally compact Hausdorff space, and let $\Gamma$ be a second-countable locally compact unimodular group ($\Delta\equiv 1$) acting on $X$ by homeomorphisms. Let $\sigma\in Z^2 (\Gamma,\T)$. Then there exists a surjective  $*$-homomorphism $\Phi :L^1(\Gamma, C_0(X),\sigma) \to  L^1(X\rtimes \Gamma, \sigma)$.  Consequently, if $L^1(\Gamma, C_0(X),\sigma)$ is Hermitian, then  so is $L^1(X\rtimes \Gamma, \sigma)$.
		 \end{lemma}
		\begin{proof}
		First observe that $C_c(\Gamma,C_c(X),\sigma)$ is  a dense $*$-subalgebra of $L^1(\Gamma,C_0(X),\sigma)$. 
		The map $\Phi:C_c(\Gamma,C_c(X),\sigma)\to C_c(X\rtimes \Gamma,\sigma)$ given by $\Phi(f)(x,\gamma)= \Phi(f)(\gamma)(x)$ defines a $*$-homomorphism. Indeed, given $f,g\in C_c(\Gamma,C_c(X),\sigma)$, $\gamma\in \Gamma$ and $x\in X$
				 
				 \begin{align*}
		\Phi(f\star_\sigma g )(x,\gamma) & = (f\star_\sigma g)(\gamma)(x) \\
		& =\int_\Gamma \sigma(\gamma\mu^{-1},\mu) f(\gamma\mu^{-1})(\mu\cdot x)g(\mu)(x)\,d\lambda(\mu) \\
		& =\int_\Gamma \sigma(\gamma\mu^{-1},\mu) \Phi(f)(\mu\cdot x,\gamma\mu^{-1})\Phi(g)(x,\mu)\,d\lambda(\mu) \\
		& =(\Phi(f)\star_\sigma \Phi(g))(x,\gamma)\,,
				 \end{align*}
				 and
				 \begin{align*}
				 \Phi(f^{*_\sigma})(x,\gamma) & =   f^{*_\sigma}(\gamma)(x)= \overline{\sigma(\gamma,\gamma^{-1} )f(\gamma^{-1})(\gamma\cdot x)} \\
				 & =  \overline{\sigma(\gamma,\gamma^{-1} )\Phi(f)(\gamma\cdot x,\gamma^{-1})} = \Phi(f)^{*_\sigma}(x,\gamma)\,.
				 \end{align*}
				  Observe that clearly $\Phi$ is a bijection. Now given $f\in C_c(\Gamma,C_c(X),\sigma)$  we have that 
				\begin{align*}
				\|\Phi(f)\|_I & = \sup \left\lbrace  \int_\Gamma |\Phi(f)(x, \gamma)| \,d\lambda(\gamma) : x\in X  \right\rbrace \\
				& = \sup \left\lbrace  \int_\Gamma |f(\gamma)(x)| \,d\lambda(\gamma) : x\in X  \right\rbrace  \\
						&  \leq \int_\Gamma \sup\{|f(\gamma)(x)|: x\in X  \} \,d\lambda(\gamma) \\
							&  = \int_\Gamma \|f(\gamma)\|_\infty  \,d\lambda(\gamma)=\|f\|\,.		 
				\end{align*}	
			 	 Thus, $\|\Phi(f)\|_I\leq \|f\|$ for every $f\in C_c(\Gamma,C_c(X),\sigma)$, and so $\Phi$ extends to a continuous surjective $*$-homomorphism $\Phi: L^1(\Gamma, C_0(X),\sigma)\to L^1(X\rtimes \Gamma,\sigma)$. 
			 	 
			 	Finally, the last statement follows by the fact that quotients of Hermitian $*$-algebras are Hermitian \cite[Theorem 10.4.4]{Palmer}. 
		\end{proof}

		\begin{corl}
Let $X$ be a second-countable locally compact Hausdorff space, and let $\Gamma$ be a second-countable compact group or a locally compact abelian group acting on $X$ by homeomorphisms. Let $\sigma\in Z^2 (\Gamma,\T)$. Then  $L^1(X\rtimes \Gamma, \sigma)$ is Hermitian, and, in particular, $L^1(X\rtimes \Gamma, \sigma)$ is spectrally invariant in $C^*_r (X \rtimes \Gamma, \sigma)$. 
		\end{corl}
		\begin{proof}
				By Example \ref{exampl_trans} we have that $(X\rtimes \Gamma)_\sigma\cong X\rtimes \Gamma_\sigma$. Then by Proposition \ref{thm_herm}  $L^1(X\rtimes \Gamma,\sigma)$ is Hermitian if $L^1(X\rtimes \Gamma_\sigma)$ is Hermitian. By Lemma 	\ref{lemma_iso} $L^1(X\rtimes \Gamma_\sigma)$ is Hermitian if $L^1(\Gamma_\sigma,C_0(X))$ is Hermitian.  Now if $\Gamma$ is compact, then $\Gamma_\sigma$ is compact too, and if  $\Gamma$ is abelian, then $\Gamma_\sigma$ is nilpotent, and hence  $L^1(\Gamma_\sigma,C_0(X))$ is Hermitian \cite[pg.\ 1285]{BeBe}. So in both cases $L^1(X\rtimes \Gamma,\sigma)$ is Hermitian, and it then follows by Theorem \ref{corl:quasi_symmetric_approx_theory} that $X \rtimes \Gamma$ has the weak containment property with respect to $\sigma$. Finally, by Corollary \ref{Wiener_env} it follows that $L^1(X\rtimes \Gamma,\sigma)$ is spectrally invariant in $C^*_r (X \rtimes \Gamma, \sigma)$.
				\end{proof}

\end{document}